\documentclass[12pt]{article}

\usepackage{amsmath}
\usepackage{amscd}
\usepackage{amsopn}
\usepackage{exscale}
\usepackage{amsthm}
\usepackage{amsfonts,bbm}
\usepackage{latexsym}
\usepackage{a4}
\usepackage[all]{xy}
\usepackage{color}
\usepackage{graphicx}

\makeindex
\makeatletter
\@addtoreset{equation}{section}
\makeatother

\newtheorem{lemma}{Lemma}[section]
\newtheorem{proposition}[lemma]{Proposition}
\newtheorem{corollary}[lemma]{Corollary}
\newtheorem{theorem}[lemma]{Theorem}
\newtheorem{remark}[lemma]{Remark}
\newtheorem{example}[lemma]{Example}

\newcommand{\real}{\mathbbm{R}}
\newcommand{\nat}{\mathbbm{N}}

\newcommand{\sumn}{\sum_{m=0}^{\infty}}
\newcommand{\sunn}{\sum_{n=0}^{\infty}}

\renewcommand{\a}{\alpha}
\renewcommand{\b}{\beta}

\newcommand{\g}{\gamma}
\renewcommand{\l}{\lambda}
\newcommand{\vp}{\varphi}
\newcommand{\ve}{\varepsilon}

\newcommand{\reald}{{\real^d}}
\newcommand{\on}{\quad\text{ on }}
\newcommand{\und}{\quad\mbox{ and }\quad}

\newcommand{\W}{\mathcal W}  
  
\newcommand{\E}{{\mathcal E}}

\newcommand{\T}{\mathcal T}

\newcommand{\Nc}{N}
\newcommand{\etac}{\eta}

\newcommand{\cM} {{\mathcal M}}
\newcommand{\cX} {{X}}

\let\oldmarginpar\marginpar
\renewcommand\marginpar[1]{\-\oldmarginpar[\raggedleft\footnotesize #1]%
{\raggedright\footnotesize #1}}

\usepackage[bookmarks,colorlinks]{hyperref}

\newcommand{\mn}{\marginpar{\ \ \ \ \ !}}

\definecolor{tj}{rgb}{0.4,0,0.6}

\definecolor{wh}{rgb}{0.7,0.25,0}

\definecolor{kb}{rgb}{0.2,0.7,0.1}

\title{Localization  and Schr\"odinger perturbations of kernels
\footnotetext{
The research was partially supported by grants MNiSW N N201 397137, MNiSW N N201 422539,
ANR-09-BLAN-0084-01.}}
\author{Krzysztof Bogdan
\thanks{Institute of Mathematics of the
Polish Academy of Sciences, ul. \'Sniadeckich 8, 00-956 Warszawa, Poland,
bogdan@pwr.wroc.pl}
\and
Wolfhard Hansen 
\thanks{Fakult\"at f\"ur Mathematik, Universit\"at Bielefeld, 
Postfach 100131,
D-33501 Bielefeld, Germany, hansen@math.uni-bielefeld.de
}
\and 
Tomasz Jakubowski
\thanks{Institute of Mathematics and Computer Science, Wroc{\l}aw University of Technology, Wybrze\.ze Wyspia\'nskiego 27, 50-370 Wroc{\l}aw, Poland,
Tomasz.Jakubowski@pwr.wroc.pl}
}
\date{\today}
\begin{document}
\maketitle
\begin{abstract} 
We study iterations of integral  kernels satisfying a transience-type condition and we prove exponential estimates analogous to Gronwall\rq{}s inequality.
As a consequence we obtain estimates of Schr\"odinger perturbations of integral kernels, including Markovian semigroups.
\end{abstract}
\section{Introduction}

To motivate our results we
consider the Gaussian transition density
on $\real^d$,  
\begin{equation*}
        p(s,x,t,y)=\begin{cases}
                     {[4\pi(t- s)]^{-d/2}} \exp
                      \dfrac{-|x-y|^2}{4(t-s)},& \mbox{ if }
                      s<t,\\
                    0,&\mbox{ if } s\ge t,
                      \end{cases}
\end{equation*} 
where $d\ge 1$, $s,t\in\real$ and $x,y\in\real^d$. Note that $-p$ is a  left inverse of $\partial_t+\Delta_y$: 
$$
\int_\real\int_\reald p(s,x,t,y)\left[\partial_t\phi(t,y)+\Delta_y \phi(t,y)\right]dydt=-\phi(s,x),\quad \phi \in C^\infty_c(\real\times \reald).
$$
Let $q(t,y)\geq 0$ be a Borel function on $\real\times \reald$. Let
$p_0=p$, and for $n=1,2,\dots$,
\begin{equation}\label{eq:pG}
p_n(s,x,t,y)=\int_\real \int_\reald p_{n-1}(s,x,u,z)q(u,z)p(u,z,t,y)dudz.
\end{equation}
We define
$
\tilde p= \sum_{n=0}^\infty p_n.
$
Under appropriate integrability  conditions, $-\tilde p$ is the left inverse of  $\partial_t+\Delta_y+q$ (\cite{MR2457489}).
We call $\tilde p$ the Schr\"odinger perturbation of $p$ by $q$, because $\partial_t+\Delta_y+q$ is  an additive perturbation of $\partial_t+\Delta_y$ by the operator of multiplication by $q$. 
We see that $\tilde p(\cdot,\cdot,t,y)$
is a power series of iterates of an integral 
kernel operator applied to 
$p(\cdot,\cdot,t,y)$, which may be considered as a {\it control function}. 

Estimates of such series for rather general kernels are the  main subject of the paper,
motivated by the results of \cite{MR2457489,MR2507445} on transition densities.
The main feature of our approach is majorization of the series by means of 
a control function, e.g. $f$ in our main result, Theorem~\ref{upper estimate}.
The assumptions on the kernel involve 
{\it local smallness} \eqref{j-local} and {\it global boundedness} \eqref{j-global}
with respect to an increasing family of {\it absorbing sets}, which add a strong transience-type property of the kernel to the picture.
A representative application of Theorem~\ref{upper estimate} is given in Example~\ref{ex:pkssss} for the potential kernel of two $1/2$-stable subordinators.

In general we neither assume Chapman-Kolmogorov conditions on the kernel
nor any connection between the kernel and the control function.
However, for Schr\"odinger perturbations,
these two are related by a multiplication operation,
and the setting of {\it space-time} 
is of special interest because it includes transition kernels. The setting is dealt with in Theorem~\ref{main-semi}, which is complemented by Example~\ref{ex:at2} and Corollary~\ref{main-semick}, and illustrated by Example~\ref{ex:Ctd}.

Our results are analogues, and a strengthening, of Khasminski\rq{}s lemma (\cite{MR1329992, MR2471145}), under a transience-type properties
of the kernel. 
They may 
be regarded as extensions of Gronwall\rq{}s    
lemma to the context of kernel operators. 
The results also apply to Schr\"odinger perturbations of continuous-time transition densities  by {\it measures}.
They may be used in discrete time, in fact in quite general settings, including partially ordered state spaces. 
In a  related paper \cite{2011-KBTJSS} we  
use different 
methods to obtain slightly more specific estimates for
Schr\"odinger perturbations of kernels on space-time by {\it functions}.

The paper is composed as follows. In Section~\ref{sec:ab} we consider integral kernels 
on {absorbing} sets. In Section~\ref{sec:mp} we prove estimates of von Neumann series for such kernels in presence of a control function. In Section~\ref{sec:td} we give 
the application to  Schr\"odinger perturbations 
of the potential kernel of two subordinators. We also discuss the local smallness and global boundedness for continuous-time kernels, with focus on transition kernels and singular perturbations, including perturbations by measures.

\section{Kernels and absorbing sets}\label{sec:ab}

Let $(E,\E)$ be  a measurable space and let $K$ be a kernel
on $(E,\E)$ (\cite{MR939365}). That is, $K:E\times \E\to [0,\infty]$, each $K(x,\cdot)$ is a measure on~$(E,\E)$,  and each function $K(\cdot,B)$
 is $\E$-measurable. We write $f\in \E^+$ if $f\colon E\to [0,\infty]$ and $f$ is $\E$-measurable. For $f\in \E^+$ we let 
\begin{equation}\label{eq:ko}
Kf(x)=\int f(y)\,K(x,dy), \quad x\in E,
\end{equation}
and call this $K$ a kernel operator. The operator is additive, positively homogeneous, %
and $Kf_n(x)\uparrow Kf(x)$ whenever $f_n\uparrow f$. 
Conversely, every map from $\E^+$ to~$\E^+$ having these properties is of the form (\ref{eq:ko}), see \cite{MR939365}.
For instance, if $q\in \E^+$,  then {\it the multiplication by $q$}, 
$$
qf(x):=q(x)f(x),\quad x\in E, f\in \E^+,
$$
is a kernel operator.  This is a simple but ambiguous notation, and it should always be clear from the context
which meaning of $q$ we have in mind (the function or the multiplication operator). 
The composition of kernel operators $K$ and $L$ on $\E^+$ and the { composition of kernels}, $KL(x,B)=\int L(y,B)K(x,dy)$ on $(E,\E)$, agree in the sense of (\ref{eq:ko}), and so the composition of kernels is associative. We will often consider the multiplication 
by $1_A$, the indicator function of \hbox{$A\in \E$}.

A set $A\in \E$ is called \emph{$K$-absorbing}, if $K(x,A^c)=0$ for every $x\in A$, that is if $1_AK1_{A^c}=0$. Since $1_E=1_A+1_{A^c}$ 
and $1_A1_{A^c}=0$, $A$ is $K$-absorbing if and only if
\begin{equation}\label{eq:AKA}
1_AK=1_AK1_A
\end{equation}
as kernels.
Clearly, $\emptyset$~and $E$ are $K$-absorbing, and the union and intersection of countably many $K$-absorbing sets are $K$-absorbing.
 If $A$ is $K$-absorbing, then $A$ is $L$-absorbing for any kernel $L\le K$.

\begin{example}\label{ex:p}{\rm
We will generalize the discussion of the Gaussian kernel from Introduction.
Let $(\cX,\cM)$ be a measurable space. Let 
$E=\real\times \cX$, with  the 
$\sigma$-algebra $\E$ generated by the sets $(a,b)\times
A$, where $a,b\in\real$, $a<b$ and $A\in\cM$.
Let $p\colon E\times E\to [0,\infty]$ be  $\E\otimes \E$-measurable and satisfy
\begin{equation}\label{zero}
              p(s,x,t,y)=0, \quad\mbox{ whenever }\quad  s\ge t.
\end{equation} 
Given a measure $\mu$ on~$(E,\E)$, we define the kernel
$K^\mu$,
\begin{equation}\label{eq:Kmu}
       K^\mu f(s,x):=\int p(s,x,u,z)f(u,z)\,d\mu(u,z), \qquad (s,x)\in
       E,\quad f\in \E^+.
\end{equation}
We note that, for every $t\in\real$, 
 the \lq\lq{}open  half-space\rq\rq{} 
$(t,\infty)\times \cX$ and the \lq\lq{}closed half-space\rq\rq{} $[t,\infty)\times \cX$ are absorbing for $K^\mu$.  Thus, the first coordinate has a distinguished role  for {\it space-time} $E=\real\times \cX$, which is the main setting of \cite{2011-KBTJSS}.

In many examples of interest $p$ also satisfies the {Chapman-Kolmogorov equations}, i.e., there is a measure $m$ on~$(\cX,\cM)$ such that
for all $s<u<t$ and $x,y\in\cX$,
\begin{equation}\label{chap-kol}
               p(s,x,t,y)=\int p(s,x,u,z)p(u,z,t,y)\,dm(z). 
\end{equation} 
For the Brownian transition density, $m$ is the Lebesgue measure
on $\real^d$.
}
\end{example}

\begin{example}\label{ex:s}
{\rm
Let $(T,\T,\rho)$ be a measure space. Let $\{K_t\ , t\in T\}$ be a family of kernels on $(E,\E)$ such that $(t,x)\mapsto K_t(x,B)$ is 
$\T\otimes \E$-measurable
 for each $B\in \E$. Then $K:=\int K_t \, \rho(dt)$ is a kernel. Furthermore, if $A\in \E$ is $K_t$-absorbing 
for every $t\in T$, then $A$ is also $K$-absorbing. 

For instance, let $\alpha\in (0,2)$ and let $p_t(y)$ be the density function of the 
$\alpha/2$-stable subordinator $(\eta_t,\ t>0)$ on $\real$.
Recall that $(\eta_t)$ is time-homogeneous and has independent increments,
 and $p_t(y)=0$ if $y\leq 0$. Thus the right half-lines are absorbing for the semigroup $K_t(x,dy):=p_t(y-x)dy$. We have (see, e.g., 
\cite[V.3.4]{MR850715} or \cite[(1.38)]{MR2569321}),
$$
\int_0^\infty p_t(y)\,dt=\Gamma(\alpha/2)^{-1}y^{\alpha/2-1}\ , \quad y>0\ . 
$$
Accordingly, the right half-lines are absorbing for the {\it potential kernel} of $(\eta_t)$,
$$
K(x,A)=\Gamma(\alpha/2)^{-1}\int_{A} (y-x)_+^{\alpha/2-1}\,dy,
$$
and also for 
$$
K^\mu(x,A)=\Gamma(\alpha/2)^{-1}\int_{A}(y-x)_+^{\alpha/2-1}\,\mu(dy),
$$
where $\mu$ is any Borel measure on $\real$.
}
\end{example}

\begin{example}\label{rem:cones}
{\rm 
If $E$ is partially ordered and each measure $K(x,dy)$ is concentrated on $\Gamma_x:=\{y:x\prec y\}$, then the sets  $\Gamma_x$ are $K$-absorbing. This is the case, e.g., for the semigroup and the potential operator of a vector of  subordinators (see also Example~\ref{ex:pkssss}).
}
\end{example}

\begin{example}\label{ex:W}
{\rm
Let $({\mathcal X},\W)$  be a balayage space (\cite[II.4]{MR850715}). Here  ${\mathcal X}$ is a locally compact space with countable base, and $\W$ denotes the class of nonnegative hyperharmonic functions on ${\mathcal X}$ (\cite[III.1]{MR850715}). In particular, each $w\in \W$ is lower semicontinuous.
Let $r$ be a continuous real
potential on~${\mathcal X}$ (\cite[II.5]{MR850715}) and let $K$ be the potential kernel associated with $r$ in the sense of  \cite[II.6.17]{MR850715}. Thus, $K1=r$, and for every bounded Borel measurable
function $f\ge 0$ on~${\mathcal X}$, the function~$Kf$ is a continuous potential, which is harmonic outside the support of~$f$, see \cite[III.6.12]{MR850715}. 
Let $w\in \W$ and $A=\{x\in {\mathcal X}: w=0\}$.
Then $A$ is closed and $K$-absorbing.
Indeed, let $B$ be a compact in~$A^c$. There exists a number $c>0$ such that $cw>r$ on $B$. By the minimum principle (\cite[III.6.6]{MR850715}), $cw\ge K1_B$ everywhere, hence
$K1_B=0$ on $A$. 
In \cite[V.1]{MR850715} such sets $A$ are called {\it absorbing}, too, and they have a number of equivalent characterizations, of which we mention two:
(a) $A$ is closed and $P_t(x,{\mathcal X}\setminus A)=0$, for every $t>0$, $x\in A$, and sub-Markov semigroup $(P_t)_{t>0}$ having $\W$ as excessive functions, and (b) $A$ is closed and $P^x[X_t\in A\cup \{\partial\}]=1$
for every $t>0$, $x\in A$,  and Markov process $(X_t,P^x)_{t>0,x\in {\mathcal X}}$ having $\W$ as excessive functions and $\partial$ as the cementary state. The details are given in \cite[V.1.2]{MR850715}.

Furthermore,  if $A$ is any Borel set containing  the (fine) superharmonic support of~$r$, then $K1_A=K$ (\cite[II.6.3]{MR850715}), and hence $A$  is $K$-absorbing.
}
\end{example}

We will collect a few simple facts about $K$-absorbing sets.
\begin{lemma}\label{triv}
Let $A$ be  $K$-absorbing and $m\in\nat$. Then 
\begin{equation}\label{aka}
1_AK^{m}=(1_AK)^m=1_AK^m1_A.
\end{equation} 
In particular, $A$ is $K^m$-absorbing. If furthermore $f\in\E^+$  and $c\ge 0$ are such that 
$Kf\le c f$ on~$A$, then $K^mf\le c^m f$ on~$A$.
\end{lemma} 

\begin{proof} The case of $m=1$ follows from (\ref{eq:AKA}). If (\ref{aka}) holds for some $m\in\nat$, then
$$ 
1_A K^{m+1}=1_AK^mK=(1_AK)^m1_AK=(1_AK)^m1_AK1_A
$$
showing that (\ref{aka}) holds for $m+1$, and we can use induction. Further, $1_AKf\leq cf$ yields that
$    1_A K^{m}f=(1_A K)^mf\le  c^m  f$.
\end{proof}

\begin{lemma}\label{difference}
Let $A$ and $B$ be $K$-absorbing, $A\subset B$, and  $m\in\nat$. Then
\begin{equation}\label{klm}
             1_BK^m1_{B\setminus A}= 1_B (K1_{B\setminus A})^m=1_{B\setminus A}(K 1_{B\setminus A})^m.
\end{equation} 
\end{lemma} 
\begin{proof} Since $A$ is $K$-absorbing, 
$1_B K 1_{B\setminus A}=1_A K 1_{B\setminus A}+
1_{B\setminus A} K 1_{B\setminus A}=1_{B\setminus A} K 1_{B\setminus A}$.
 By this and Lemma \ref{triv}  (with $B$ in place of $A$), 
\begin{align*}
1_B K^m 1_{B\setminus A}&=
(1_B K)^m1_{B\setminus A}=
(1_B K)^{m-1}1_{B\setminus A}K1_{B\setminus A}
=\ldots=1_{B}(K 1_{B\setminus A})^m \\
&
=1_{B}K1_{B\setminus A}(K 1_{B\setminus A})^{m-1}
=1_{B\setminus A}K1_{B\setminus A}(K 1_{B\setminus A})^{m-1}
=1_{B\setminus A}(K 1_{B\setminus A})^m.
\end{align*}
\end{proof}

The next result is a slight modification of \cite[Proposition 7.4]{MR2058029}.
\begin{proposition}\label{series-equiv}
Let $A$ be $K$-absorbing, and let $f\in\E^+$ and $c\ge 1$ be such that 
$\sum_{m=0}^\infty K^m f\le c f$ on~$A$. Then, for $n={0},1,\ldots$, we have
\begin{equation}\label{eq:eKn}
K^n f\le c\bigl(1-1/c\bigr)^n f \on A.
\end{equation}
\end{proposition} 
\begin{proof} 
Let $g=\sum_{m=0}^\infty K^m f$. We see that $g=f+Kg\ge (1/c) g + Kg$ on~$A$, hence $Kg\le (1-1/c)g$ on~$A$. By Lemma \ref{triv}, for every $n\in\nat$,
$$
      K^n f\le K^n g\le \bigl(1-1/c\bigr)^n g
\le c\bigl(1-1/c\bigr)^n f \on A.
$$
The case of $n=0$ is trivial. 
\end{proof}
\begin{remark}\label{rem:an} 
{\rm 
We note that, conversely, (\ref{eq:eKn}) yields that 
$$\sum_{n=0}^\infty K^n f\leq  
\sum_{n=0}^\infty c (1-1/c)^n f= c^2 f \on  A. 
$$
Thus, comparability of $\sum K^n f$ and $f$ is equivalent to exponential decay of $K_n f$.
}
\end{remark}
\begin{remark}\label{rem:an1} 
{\rm 
We will consider $f=1$, the constant function. For every $a\geq 1$,
there exist kernels $K$ such that $\sup_{x\in E} K1(x)=a$, but  $\sumn K^m1$ is bounded (see \cite[Proposition 10.1]{MR2207878}). 
Then the estimate for $K^n1$ given in (\ref{eq:eKn}) is asymptotically better than the more evident upper bound by $a^n$. 
}
\end{remark}

\section{Localization on differences of absorbing sets}\label{sec:mp}

We first prove a discrete variant of Gronwall\rq{}s lemma.
\begin{lemma}\label{gronwall}
Let $\a,\delta\in [0,\infty)$ and $\g_1,\dots,\g_k\in\real$ be such that for $j=1,\dots,k$, we have
$
\g_j\le \a+\delta\!\sum\limits_{1\le i< j}\g_i
$.
Then  $\g_j\le \a (1+\delta)^{j-1}$  for every $j=1,\dots,k$. 
\end{lemma} 
\begin{proof} We proceed by induction:         
$\g_{k+1}\le \a +\delta\sum_{i=1}^k\a(1+\delta)^{i-1}=\a(1+\delta)^k$.
\end{proof} 
%
%For this section w
We fix $K$-absorbing sets $A_1,\dots,A_k$ such that
$$
            A_1\subset A_2\subset \dots \subset A_k.
$$
Taking $A_0:=\emptyset$, for $1\le j\le k$ we define
{\it slices} $S_j:=A_j\setminus A_{j-1}$ and  operators
$$
K_j:=K 1_{S_j}.
$$
Thus, in Example \ref{ex:p} we may choose {$-\infty<t_k<\dots<t_1<\infty$}, and let $A_j$ be the open half-space 
$(t_j,\infty)\times \cX$ or the closed half-space $[t_j,\infty)\times \cX$. Then
each slice $S_j$ equals $I_j\times \cX$, 
where $I_j$ is an interval,
see also Example~\ref{ex:pkssss} and Figure~\ref{pic:cw}.

\begin{theorem}\label{upper estimate}
Let $0\le \eta<1$,  $\beta\ge 0$,
and  $f\in\E^+$ be such that
\begin{equation}\label{j-local}
           K_jf\le \eta f  \quad \mbox{on } S_j,\quad j=1,\ldots,k,
\end{equation} 
and
\begin{equation}\label{j-global}
           K_jf\le \beta f  \quad \mbox{on } A_k, \quad j=1,\ldots,k.
\end{equation} 
Then, for  $j=1,\ldots,k$,
\begin{equation}\label{main}
\sumn K^m f\le  \frac1{1-\eta}\left(1+\frac{\beta}{1-\eta}\right)^{j-1}\,f
\on S_j.
\end{equation} 
\end{theorem} 
\begin{proof}Let $n\in\nat$ and $   g_n:=\sum_{m=0}^n K^mf$.
For $j=1,\ldots,k$, we (recursively) define
\begin{equation}\label{g-rec}
          \g_j:=\frac 1{1-\eta}\bigl(1+\beta\sum_{1\le i<j}\g_i\bigr).
\end{equation} 
We will prove by induction that $g_n\le \g_j f$ on $S_j$.
Let $1\le j\le k$, and
\begin{equation}\label{induction}
  g_n\le \g_i f \on S_i, \qquad\mbox{  for every }1\le i<j.
\end{equation} 
Trivially, this assumption is satisfied for $j=1$. 
By (\ref{j-global}),  $Kf\le k\b f$ on~$A_k$.  By Lemma~\ref{triv} we obtain a rough bound, $ g_n\le  \sum_{m=0}^n (k\b)^m f$ on $A_k$.
Let $\g\geq 0$ be the smallest real number such that $g_n \le \gamma f$ on~$S_j$. 
If $j<l\le k$, then $1_{S_j}K_l=0$. By  (\ref{induction}) and (\ref{j-global}) for all $x\in S_j$ we have, 
\begin{eqnarray*} 
   g_n(x)&\le& f(x)+Kg_n(x)=f(x)+\sum_{i=1}^j K_i g_n(x)\\
&\le&f(x)+\sum_{i=1}^{j-1} \g_i K_if(x)+\g
K_jf(x)
\le \bigl(1+\beta\sum_{i=1}^{j-1} \g_i\bigr)f(x) +\g \eta f(x).
\end{eqnarray*} 
Thus $\g\le \g_j$ (see  (\ref{g-rec})),
$g_n\le \g_j f$ on~$S_j$, and the result follows by Lemma~\ref{gronwall}.
\end{proof} 

\begin{remark}\label{sje} {\rm We shall refer to (\ref{j-local}) as  {\it local smallness} and to (\ref{j-global}) as {\it global boundedness}. 
In many important cases, the local smallness already implies the global boundedness with $\beta=\eta$. In particular, it is so in Example \ref{ex:W}, if   $f,1\in \W$.
This follows from the minimum principle \cite[III.6.6]{MR850715} applied to the functions $\eta f-K1_L\min\{f,n\}$, for compacts sets $L\subset S_j$ and
$n\in\nat$.
It is also true in Example \ref{ex:p} provided  
$f=p(\cdot,\cdot,t,y)$,
each $A_j$~is a~half-space, $\mu$ does not charge the ``hyperplanes'' $\{t\}\times \cX$, $t\in \real$, and the Chapman-Kolmogorov equations are satisfied, see Lemma {\ref{p-local-1}} below.
}
\end{remark}

The following result is motivated by Proposition~\ref{series-equiv} and Remark~\ref{rem:an}.
\begin{corollary}\label{main'}
Assume $c> 1$ and ${\Nc} \in\nat$ are such that ${\etac} := c\left(1-1/c\right)^{\Nc }<1$.  Let $\beta\ge 0$ and  $f\in\E^+$ be such that
$Kf\le \beta f$ on~$A_k$ and
\begin{equation}\label{local}
\sumn K_j^mf\le cf \on S_j
\end{equation} 
for every $1\le j \le k$.
Then, for every $1\le j \le k$,
\begin{equation}\label{eq:Kmct}
       \sumn K^mf\le \left(\sum_{n=0}^{\Nc -1} \beta ^n\right)
\frac 1{1-\etac }
\left(1+\frac{\beta}{1-\etac }\right)^{j-1}  \, f  \on  S_j.
\end{equation} 
If {\rm (\ref{local})} holds on~$A_k$ for every $1\le j \le k$, then $Kf\le ck f$ on~$ A_k$.
\end{corollary}

\begin{proof}
 Let $1\le j\le k$. Since each $K_j^mf$ vanishes on~$A_{j-1}$, 
(\ref{local}) means that $\sumn K_j^m f\le cf$ on~$A_j$.
By a remark following (\ref{eq:AKA}),  $A_j$ is $K_j$-absorbing. By Lemma \ref{difference} and Proposition \ref{series-equiv},
$$
(K^{\Nc })(1_{S_j}f)=(K_j)^{\Nc } f\le \etac  f \on A_j.
$$
An application of  Theorem~\ref{main} yields that
$$
           \sumn (K^{\Nc })^m f\le \frac 1{1-\etac }
\left(1+\frac{\beta}{1-\etac }\right)^{j-1} f \on S_j.
$$
By Lemma \ref{triv}, 
$\sum_{n=0}^{\Nc -1} K^n f\le \sum_{n=0}^{\Nc -1} \beta^n f$ on~$A_k$.
We finally note that
$$
  \sumn K^m f= \sumn (K^{\Nc })^m\left(\sum_{n=0}^{\Nc -1} K^n f\right)
             \le \left(\sum_{n=0}^{\Nc -1} \beta^n \right)\sumn (K^{\Nc })^m f.
$$
If (\ref{local}) holds even on~$A_k$ for $1\le j \le k$, then $Kf\le \sum_{j=1}^k K_j f\le ck f$ on~$ A_k$, {and we can take $\beta=ck$}.
\end{proof}

\section{Examples and Applications}\label{sec:td}
We may use Theorem~\ref{upper estimate} to estimate Schr\"odinger-type perturbations of kernels.
As a rule, auxiliary estimates of the kernels are needed for 
such applications. 
\begin{example}\label{ex:pkssss}
{\rm
For $t>0$ and $x\in \real$ we define
$$
f_t(x)=\left\{
\begin{array}{cc}
(4\pi)^{-1/2}\, t\, x^{-3/2}e^{-t^2/(4x)},  &\mbox{ if $x>0$,}\\
0\, ,& \mbox{ else, }
\end{array}
\right.
$$
%$f_t$ is 
the density function of the $1/2$-stable subordinator. By \cite[Example~2.13]{MR1739520},
$$
\int_0^\infty f_t(x)e^{-ux}dx=e^{-t u^{1/2}}\, , \quad u\ge 0\, .
$$
For $\phi\in C_c^\infty(\real)$ (smooth compactly supported real-valued functions on $\real$) we let 
%consider the semigroup
$$
P_t\phi(x)=\int_0^\infty \phi(x+z)f_t(z)dz\ , \quad x\in \real\, .
$$ 
The generator of the semigroup $(P_t)$ is the Weyl fractional derivative,
\begin{align*}
\partial^{1/2}\phi(x) &=  
\int_0^\infty (4\pi)^{-1/2}\, z^{-3/2}\left(\phi(x+z)-\phi(x)\right)dz\\
&=\pi^{-1/2}\int_0^\infty z^{-1/2}\phi'(x+z)\,dz\,.
\end{align*}
Schr\"odinger perturbations of $\partial^\beta$ for $\beta\in (0,1)$ were considered in \cite{2011-KBTJSS}. We shall
discuss those for
%consider 
the generator $L=\partial^{1/2}_s+\partial^{1/2}_x$ of the semigroup of two independent $1/2$-stable subordinators,
$$
T_t\vp(s,x)=\int_0^\infty \int_0^\infty \vp(s+u,x+z)f_t(u)f_t(z)dudz\, ,\quad s,x\in \real \, .
$$
Here and below, $\vp\in C_c^\infty(\real\times \real)$. For $s,x\in \real$ we have
\begin{align}
\vp(s,x)&=-\int_0^\infty \frac{d}{dt}T_t \vp(s,x)\,dt
=-\int_0^\infty T_t\  L\vp(s,x)\,dt\, . \label{eq:pwrc}
\end{align}
In view of (\ref{eq:pwrc}) we need to calculate the potential kernel $\int_0^\infty T_t dt$. Let
\begin{align*}
\kappa(s,x)&=\int_0^\infty f_t(s)f_t(x)dt\\
&=\left\{
\begin{array}{cc}
(4\pi)^{-1/2}(s+x)^{-3/2}\, ,& \mbox{ if $s,x>0$,}\\
0\, ,& \mbox{else,} 
\end{array}
\right.
\end{align*}
where the latter formula follows from direct integration. Define
$$
\kappa(s,x,u,z)=\kappa(u-s,z-x)\,, \quad s,x,u,z\in \real \, .
$$
By \eqref{eq:pwrc}, we obtain
\begin{align}\label{eq:fs2fd}
\int_\real\int_\real \kappa(s,x,u,z) (\partial^{1/2}_u+\partial^{1/2}_z)\vp(u,z)\,dudz=-\vp(s,x)\,,\quad s,x\in \real\,.
\end{align}
We observe a 3G-type inequality: if $s< u< t$ and $x< z< y$, then
\begin{equation}\label{eq:3Gsss}
\kappa(s,x,t,y)\le \kappa(s,x,u,z)\wedge \kappa(u,z,t,y)\le 2\sqrt{2}\,\kappa(s,x,t,y)\,,
\end{equation}
since $t-s+y-x\ge (u-s+z-x)\vee(t-u+y-z)\ge (t-s+y-x)/2$. 
Thus,
\begin{align}
\kappa(s,x,u,z)\kappa(u,z,t,y)&\le 2\sqrt{2}\, \kappa(s,x,t,y)\left[\kappa(s,x,u,z)\vee \kappa(u,z,t,y)\right]\nonumber\\
\end{align}
where  $s< u< t$, $x< z< y$, and this is sharp, since \eqref{eq:3Gsss} also yields 
\begin{align}
\kappa(s,x,u,z)\kappa(u,z,t,y)&\ge 
\kappa(s,x,t,y)\left[\kappa(s,x,u,z)+ \kappa(u,z,t,y)\right]/2\,.\nonumber
\end{align}
For $0<p<1/2$ and number $c>0$ we let
$$q_0(u,z) = \left\{\begin{array}{cc}c(u+z)^{-p}& \mbox{ if } u,z >0\,,\\
0& \mbox{ else.}
\end{array}\right.
$$
We consider $0\le q\le q_0$ and  the kernel
$$Kf(s,x):=\int_{\real^2} \kappa(s,x,u,z)q(u,z)f(u,z)dzdu\,.
$$
We will use Theorem~\ref{upper estimate} to compare $\kappa$ with $\tilde \kappa$
defined as 
\begin{equation}\label{eq:deftk}
\tilde \kappa= \sumn (\kappa q)^m \kappa,
\end{equation}
or, more precisely,
$$
\tilde{\kappa}(s,x,t,y)=\sumn K^m f(s,x)\, ,
$$
where we fix $t,y\in \real$ and denote (the control function),
$$
f(s,x):=\kappa(s,x,t,y)\,.
%,\quad s,x \in \real \,.
$$
We let $s<t$ and $x<y$, because otherwise $\tilde \kappa(s,x,t,y)=0=\kappa(s,x,t,y)$. 
Furthermore, we assume that $t+y>0$, else $\tilde \kappa(s,x,t,y)=\kappa(s,x,t,y)$. Let $h>0$
and $k\ge 1$ be such that $(k-1)h\le t+y<kh$ ($h$ is defined later on). For $j=0,\ldots,k$, we let $a_j=(k-j)h$. For $j=1,\ldots,k-1$, we define $A_j=\{(u,z): u+z\ge a_j\}$.
We also let $A_0=\emptyset$, and $A_k=\real^2$. The sets $A_j$ are increasing and absorbing. 
For $j=1,\ldots, k$, we define
$S_j = A_j \setminus A_{j-1}$, 
see Figure~\ref{pic:cw}. 
\begin{figure}
\caption{Notation for Example~\ref{ex:pkssss}.}\label{pic:cw}
\includegraphics[width=15.3cm, height=9cm]{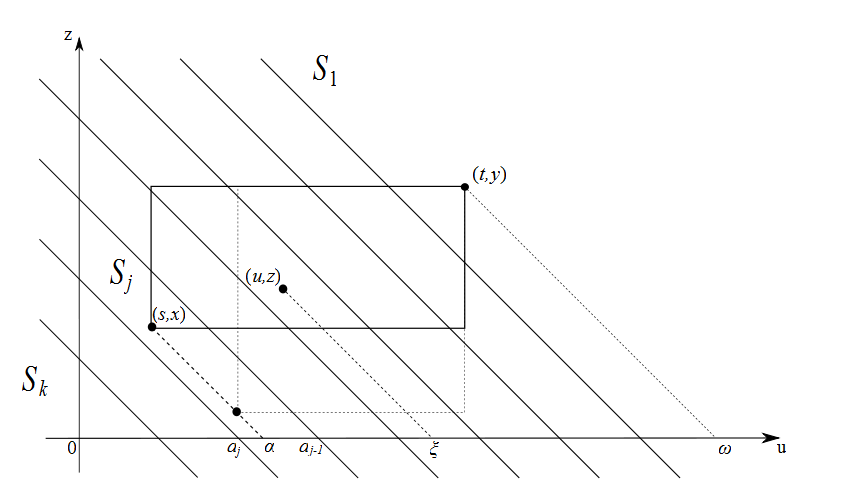}
\end{figure}
We will call $\{(u,z): u+z=\xi\}$, $\xi\in \real$,  the {\it level lines}. We define $K_j=K1_{S_j}$, as in Theorem~\ref{upper estimate}.
We have
\begin{align}\label{eq:Kjff}
& K_j f(s,x)/ f(s,x) \\
 & \le 2\sqrt{2}c \!\!\!\!\!\! \int\limits_{S_j\cap \{s \le u \le t,\; x \le z \le y\}} \!\!\!\!\!\!
[(u+z-s-x)^{-3/2} + (t+y-u-z)^{-3/2}] (u+z)^{-p} dzdu\,.\notag
\end{align}
We will estimate the right-hand side of \eqref{eq:Kjff}. 
Denote $\alpha=s+x$,  $\omega=t+y$ and $\xi=u+z$. 
Let $\alpha<a_{j-1}$ and $\omega>a_j$ (otherwise the integral is zero).
The integrand is constant along the level lines.
The integral is the largest when $\{(s,u)\in \real^2:\,s\le u \le t,\; x \le z \le y\}$ is a square, because the square's intersections with the level lines have the largest length, namely $\sqrt{2}[(\xi-\alpha)\wedge (\omega-\xi)]$, see Figure~\ref{pic:cw}.
Taking this into account or substituting $\xi=u+z, \eta=(u-z)/2$,  we bound the integral in \eqref{eq:Kjff} by
\begin{align*}
&\int_{\alpha\vee a_j}^{\omega\wedge a_{j-1}} (\xi-\alpha)(\xi-\alpha)^{-3/2}\xi^{-p}+(\omega-\xi)(\omega-\xi)^{-3/2}\xi^{-p}\; d\xi\\
&\le \int_{a_j}^{a_{j-1}} (\xi-a_j)^{-1/2-p}+ (a_{j-1}-\xi)^{-1/2}(\xi-a_{j})^{-p}\; d\xi\\
&=\big[B(1/2-p,1)+B(1/2,1-p)\big](a_{j-1}-a_j)^{1/2-p}\,,
\end{align*}
where $B$ is the Euler beta function. 

By Theorem~\ref{upper estimate}, if we let
$\eta=\beta=2\sqrt{2}c[B(1/2-p,1)+B(1/2,1-p)]h^{1/2-p}<1$ (the inequality determines $h$), then
\begin{align}
\tilde \kappa(s,x,t,y)\le \left(\frac1{1-\eta}\right)^j \kappa(s,x,t,y) \quad \mbox{for \; $(s,x)\in S_j$.}
\end{align}
In fact, $j<k+1-(s+x)/h\le (t+y-s-x)/h+2$. 
We see that $\kappa$ and $\tilde \kappa$ are locally comparable.
We also note that the first coordinate does not play a distinguished role here, in contrast to the examples in \cite{2011-KBTJSS} and below.
Finally, $\tilde\kappa$ may be considered a Schr\"odinger perturbation of $\kappa$, because
\begin{equation}\label{eq:fstsss}
\int_{\real\times\real}\tilde\kappa(s,x,u,z)
\left[\partial^{1/2}_z+\partial^{1/2}_u+q(u,z)\right] \vp(u,z)\,dzdu=-\vp(s,x),
\end{equation}
for $s,x\in \real$ and 
$\phi\in C^\infty_c(\real\times\real)$.
The identity \eqref{eq:fstsss} is proved by using \cite[(31)]{2011-KBTJSS}. 
Indeed, the absolute integrability of the integrals in \cite[(31)]{2011-KBTJSS} follows by considering the supports of the involved functions (we leave details to the reader).
We also wish to note that if $q_0(u,z)$ depends only on $u$ or $u\wedge z$,  then it is more convenient to consider absorbing sets
$\{(u,z)\in \real^2: u> s\}$ or $\{(u,z)\in \real^2: u>s, z>x\}$, correspondingly.}
\end{example}
In the remainder of the paper  we shall 
adopt the setting of Example~\ref{ex:p}. More precisely, we
consider the space-time $E=\real\times \cX$, with the product $\sigma$-algebra~$\E$, and an $\E\times\E$-measurable
function $p\ge 0$ on~$E\times E$ such that (\ref{zero}) holds, but we do {\it not} assume \eqref{chap-kol}. For
a measure $\mu$ on~$(E,\E)$ we define kernel
$K^\mu$ by (\ref{eq:Kmu}).
Motivated by the discussion 
in Introduction and Example~\ref{ex:p}, we let
\begin{equation}\label{eq:pmu}
p^\mu=\sum_{n=0}^\infty p^\mu_n,
\end{equation}
where $p_0^\mu=p$, and the positive functions $p_1^\mu, p_2^\mu\dots$  on~$E\times E$ are defined as follows,
\begin{align}\label{eq:idpn}
        p_n^\mu(s,x,t,y)&:=\int p^\mu_{n-1}(s,x,u,z) p(u,z,t,y)\,d\mu(u,z).
\end{align}
By induction, $p_n^\mu(s,x,t,y)=0$ for  $n\ge 0$, $(s,x),(t,y)\in E$, if $s\ge t$.
According to Introduction, we perturb $p$ by the {\it measure} $\mu$ (but see Example~\ref{ex:at2}, too).
We regard $(t,y)$ as fixed when iteratively transforming $f(s,x):=p(s,x,t,y)$  by $K^\mu$: 
$$
p_n^\mu(\cdot,\cdot,t,y)=(K^\mu)^n p(\cdot,\cdot,t,y).
$$
\begin{remark}
\rm
Similar perturbations may be studied for signed measures, say $\nu$.
We clearly have $|p^\nu|\leq p^\mu$, where $\mu=\nu_-+\nu_+$ is the variation measure of $\nu$.
We will not further concern ourselves with signed kernels or functions in this paper.
\end{remark}

In Example \ref{no-atoms}, \ref{ex:Dirac} and \ref{ex:at2} we will additionally suppose that $p$ is a transition density, that is, 
 the Chapman-Kolmogorov equations (\ref{chap-kol}) hold with respect to a $\sigma$-finite measure $m$ on $\cX$.

\begin{example}\label{no-atoms}
\rm 
 Let $\rho\ge 0$ be a Radon measure on $\real$ having no atoms,
and let $\mu:=\rho\otimes m$. Then, for all $(s,x),(t,y)\in E$ and $n\in\nat$,    $ p_n^\mu(s,x,t,y)=  \rho((s,t))^n p(s,x,t,y)/n!$ by induction,
and we obtain transition density
\begin{equation}\label{rho-est}
      p^\mu(s,x,t,y)=  e^{\rho((s,t))} p(s,x,t,y).
\end{equation}
\end{example}

\begin{example}\label{ex:Dirac} 
{\rm 
Let $\eta>0$, $u_0\in \real$,            
$\mu:=\eta  \ve_{u_0}\otimes m$. Here $\ve_{u_0}(f)=f(u_0)$ is the Dirac measure.
 Then $\mu$ is concentrated   on the ``hyperplane'' $\{u_0\}\times E$, and for  
$(s,x), (t,y)\in E$ we have by (\ref{chap-kol}),
\begin{equation*} 
p_1^\mu(s,x,t,y)=\int p(s,x,u,z)p(u,z,t,y)\, d\mu(u,z) =\begin{cases} 
         \eta       p(s,x,t,y), & \mbox{ if } s<u_0<t,\\[1mm]
                      0,& \mbox{ otherwise.}
                                                            \end{cases}
\end{equation*} 
For $n=2,3,\ldots$ and all $(s,x), (t,y)\in E$, we obtain $p_n^\mu(s,x,t,y)=0$, hence
\begin{equation}\label{eq:1-eta} 
p^\mu(s,x,t,y):=\sunn p_n^\mu(s,x,t,y)=
\begin{cases} 
              {(1+\eta)}\,p(s,x,t,y), & \mbox{ if } s<u_0<t,\\[1mm]
                      p(s,x,t,y),& \mbox{ otherwise.}
                                                            \end{cases}
\end{equation} 
}
\end{example}
There is, however, an alternative approach to perturbations by such measures.
\begin{example}\label{ex:at2}  
{\rm 
Let $u_0\in \real$ and    
$\mu:=  \ve_{u_0}\otimes m$. 
For $g\in \E^+$ we define
$$
Kg(s,x)=\left\{\begin{array}{lr}
0\,, & \mbox{ if } s>u_0\,,\\
g(s,x)\,, & \mbox{ if } s=u_0\,,\\
\int_{\real^d}p(s,x,u_0,z)g(u_0,z)\,dm(z)\,,& \mbox{ if } s<u_0\,.\\
\end{array} \right.
$$
Let $t>u_0$ and $y\in \real^d$ be fixed. We consider $f(s,x)=p(s,x,t,y)$, $(s,x)\in E$.
By Chapman-Kolmogorov equations, $Kf(s,x)= 1_{s\le u_0}\, p(s,x,t,y)$. 
By induction,
$K^n f(s,x)= 1_{s\leq u_0}\,p(s,x,t,y)$, for $n=1,2,\ldots$. 
If $0<\eta<1$, then
\begin{equation}\label{eq:Kq}
\tilde p(s,x,t,y):=\sum_{n=0}^\infty (\eta K)^n f(s,x)
=
\left\{\begin{array}{lr}
(1-\eta)^{-1}p(s,x,t,y)\,, & \mbox{ for } s\leq u_0\,,\\
p(s,x,t,y)\,, & \mbox{ otherwise},
\end{array}\right.
\end{equation}
whereas $\eta\geq 1$ leads to 
explosion of $\tilde p$.
We observe that $\tilde p$ satisfies Chapman-Kolmogorov equations, 
but not $p^\mu$ defined in Example~\ref{ex:Dirac}.

More generally, for an arbitrary Radon measure $\rho$ on $\real$, we let
$$Kg(s,x)=\rho(\{s\})g(s,x)+\int_{(s,\infty)}\int_\cX p(s,x,u,z)g(u,z) dm(z)\rho(du) .$$
We note that $K=K^{\rho\otimes m}$ (see (\ref{eq:Kmu})), if $\rho$ has no atoms. On one hand this motivates our interest in $K^\mu$ later in this section.
On the other hand, atoms are intrinsically related to the estimates obtained in \cite{MR2507445,2011-KBTJSS} and in Theorem~\ref{main-semi} below, because they produce inflation of mass very close to that given by the estimates. Indeed, let us fix numbers $u_1<u_2<\ldots<u_k$, and let
$\rho=\ve_{u_1}+\ve_{u_2}+\ldots+\ve_{u_k}$.
Assume that $u_k<t$. 
We have $Kf(s,x)=L(s) p(s,x,t,y)$, with $f$ as before and
$$L(s):=\#\{1\leq i\leq k:\ u_i\geq s\}.$$
By induction we verify that 
\begin{align}\nonumber
K^nf(s,x)&=\#\{(i_1, \ldots,i_n):\ s\le u_{i_1}\le \ldots\le   u_{i_n}\}p(s,x,t,y)\\
&={{L(s)+n-1}\choose{n}} p(s,x,t,y).\label{eq:ptk}
\end{align}
Notably, a similar combinatorics is triggered by gradient perturbation series in \cite[Lemma~5]{MR2876511}.
If $0<\eta<1$, then, by Taylor series expansion (\cite[p. 51]{MR2507445}), 
\begin{equation}\label{eq:iee}
\tilde p(s,x,t,y):=\sum_{n=0}^\infty (\eta K)^nf(s,x,t,y)=\left(\frac{1}{1-\eta}\right)^{L(s)}p(s,x,t,y).
\end{equation}
This should be compared with Theorem~\ref{main-semi} below.    
}
\end{example}

We now return to functions $p$ as specified before \eqref{eq:pmu}, i.e. we do not assume Chapman-Kolmogorov conditions, unless we explicitly say otherwise.

Let $I\subset \real$ be an interval and let $$\mu_I(A):=\mu(A\cap (I\times \cX)), \quad A\in \E.$$
For $n=0,1,2,\dots$, we denote (see above in this section)
$$
         p_n:=p_n^\mu  \und 
         p_n^I:=p_n^{{\mu_I}}.
$$
We also note that $p_n(s,x,t,y)=  p_n^{(s,t)}(s,x,t,y)$, which follows by induction.

The half-spaces
$(t,\infty)\times \cX$ and $[t,\infty)\times \cX$ 
are $K^\mu$-absorbing for $t\in \real$. The differences of such sets
are of the form $I\times \cX$, where $I$~is an  interval. For $I,J\subset \real$,
we write $I\prec J$, if $s<t$ for all $s\in I$ and $t\in J$.

\begin{theorem}\label{main-semi}
Let $-\infty< r<t<\infty$, $y\in \cX$, $\eta\in [0,1)$. Suppose that $[r, t)$ is the union of intervals $I_k\prec\dots\prec I_1$, such that
for all $j=1,\ldots,k$, and $x\in \cX$,
\begin{equation}\label{j-ass}
\int_{I_j\times \cX} p(s,x,u,z)p(u,z,t,y)\,d\mu(u,z) 
\le     \eta  \,  p(s,x, t, y),\quad r\leq s <t.
\end{equation} 
Then, for $j=1,\ldots,k$ and $x\in \cX$,               
\begin{equation}\label{eq:pt1-eta}
p^\mu(s,x,t,y):=\sunn p_n(s,x,t,y)
\le \left(\frac{1}{1-\eta}\right)^j p(s,x,t,y),\quad s \in I_j.
\end{equation}
\end{theorem}
\begin{proof}
We may apply Theorem~\ref{upper estimate} to $f(s,x):=p(s,x,t,y)$, $A_j=(I_j\cup\ldots\cup I_1)\times \cX$, $K_j:=K^{\mu_{I_j}}$, and $\beta:=\eta$, since (\ref{j-ass}) implies both (\ref{j-local}) and (\ref{j-global}). 
\end{proof}

\begin{corollary}\label{main-semi'}
Let $-\infty<r<t<\infty$, $y\in \cX$, $\beta\ge 0$ and $c\ge 1$ .
Suppose that 
\begin{equation}\label{eq:gb}
p_1(s,x,t,y)\le \beta\, p(s,x,t,y), \quad \mbox{ for all } \, s>r,\, x\in \cX,
\end{equation} 
and $[r,t)$ is a union of disjoint intervals $I_1,I_2,\dots, I_k$ satisfying,
\begin{equation}\label{pin}
 \sunn p_n^{I_j}(s,x,t,y)\le c\, p(s,x,t,y),  
                   \quad \mbox{ for }\   s\in I_j,\;  x\in \cX \quad (1\le j\le k).
\end{equation} 
Then there exists a constant $C$ such that $\sunn p_n(s,x,t,y)\le C\, p(s,x,t,y)$ for all $s\geq r$ and $x\in \cX$.
\end{corollary} 
\begin{proof}
We proceed as in the proof of Theorem~\ref{main-semi}, using
Corollary \ref{main'}. We let $C=\left(\sum_{n=0}^{\Nc -1} \beta^n\right)
\left[1+\beta/(1-\etac)\right]^{k-1}/(1-\etac)$, where $\eta=c(1-1/c)^N<1$. 
\end{proof}
\begin{remark}{\rm  If the inequality 
in (\ref{pin}) holds on $[r,\infty)\times \cX$,  for $1\le j\le k$, then
$$
p_1(s,x,t,y)=\sum_{j=1}^k p_1^{I_j}(s,x,t,y)
\le kc\, p(s,x,t,y), \quad s\geq r,\ x\in \cX,
$$
and (\ref{eq:gb}) holds with $\beta=kc$.
}
\end{remark}

If  $p$ satisfies  (\ref{chap-kol}), then we can {\it localize} (\ref{j-ass}) as follows.

\begin{lemma}\label{p-local-1}
Suppose that $p$ satisfies   the Chapman-Kolmogorov
equations. Let $(t,y)\in E$, $\eta\ge 0$, and let an interval $I\subset (-\infty, t)$ satisfy, for all $(s,x)\in I\times \cX$, 
\begin{equation}\label{semi-diff-form}
\int p(s,x,u,z)p(u,z,t,y)\, d\mu_I(u,z) \le \eta\, p(s,x,t,y).
\end{equation} 
Then {\rm(\ref{semi-diff-form})} holds for {\rm all} $(s,x)\in E$.
\end{lemma} 
\begin{proof}
If $s\in I$ or $s$ is to the right of $I$, then (\ref{semi-diff-form}) clearly holds, see~(\ref{zero}). If $s$ is to the left of $I$, $a\in I$, $J:=[a,\infty)\cap I$, and $x\in E$, then by (\ref{chap-kol})  and (\ref{semi-diff-form}),
\begin{eqnarray*} 
&&\int  p(s,x,u,z)p(u,z,t,y)\, d\mu_{J}(u,z)\\ 
&=&\int \int p(s,x,a,w)p(a, w,u,z)
               p(u, z,t,y)\,dm(w)\,d\mu_{J}(u, z)\\
&\le&\eta \int p(s,x,a,w) 
              p(a,w,t,y)\,dm(w)=\eta\, p(s,x,t,y).
\end{eqnarray*}
So (\ref{semi-diff-form}) holds,                                                                                            
if $\inf I\in I$ (take $a=\inf I$). If not, it follows by  monotone convergence, by letting $a\in I$  approach $\inf I$.
\end{proof} 

\begin{lemma}\label{p-local-2}
Suppose that $p$ satisfies  the Chapman-Kolmogorov
equations. Let $\eta\ge 0$ and an interval $I$ be such that,                         
for all $s,t\in I$ and $x,y\in \cX$,  
\begin{equation}\label{semi-diff-form-2}
\int p(s,x,u,z)p(u,z,t,y)\, d\mu_I(u,z) \le \eta\, p(s,x,t,y).
\end{equation} 
Then {\rm(\ref{semi-diff-form-2})} holds for {\rm all} $(s,x),(t,y)\in E$.
\end{lemma}

\begin{proof} Let us fix $(t,y)\in E$.  By~(\ref{zero}) we may replace $I$ by $I\cap (-\infty,t)$. An application of Lemma \ref{p-local-1}
finishes the proof.
\end{proof} 

\begin{corollary}\label{main-semick}          
Suppose that $p$ satisfies the Chapman-Kolmogorov equations. 
Let $-\infty< r<t<\infty$, $y\in \cX$ and  $\eta\in [0,1)$. Let $[r,t)$ be
the union of intervals $I_k\prec\dots\prec I_1$.  Assume that for $j=1,\ldots,k$ and $I:=I_j$, {\rm (\ref{semi-diff-form-2})} holds for all $s\in I_j$ and $x\in \cX$. Then {\rm (\ref{eq:pt1-eta})} holds for $j=1,\ldots,k$ and $x\in \cX$.
\end{corollary} 

\begin{proof}
The result follows from Theorem~\ref{main-semi} and Lemma~\ref{p-local-2}.
\end{proof}

\begin{remark}\label{rem:atoms}
{\rm 
To prove comparability of $p$ and $p^\mu$ under (\ref{chap-kol}) in specific situations, it is enough to choose intervals $I_j$ such that $\mu(E\setminus (I_1\cup\ldots\cup I_k)\times \cX)=0$, and for all $s,t\in I$, $x,y\in \cX$, $j=1,\ldots,k$,
\begin{equation}\label{eq:spq}
\int_{I_j} p(s,x,u,z)p(u,z,t,y)\, d\mu(z) \le \eta\, p(s,x,t,y).
\end{equation} 
\mn
%By Corollary~\ref{p-local-2} and Theorem~\ref{main-semi} we obtain.
If (\ref{semi-diff-form-2}) fails, then $p^\mu$ may be much bigger than $p$, see Example~\ref{ex:at2}.
}
\end{remark}
Our last example is essentially from \cite{MR2457489}.
\begin{example}\label{ex:Ctd}
{\rm 
We 
consider the Cauchy transition density on $\reald$, i.e. we let
\begin{equation*}
        p(s,x,t,y)=\begin{cases}
c_d (t-s)\big[(t-s)^2+|y-x|^2\big]^{-(d+1)/2}
                    ,& \mbox{ if }
                      s<t,\\
                    0,&\mbox{ if } s\ge t.
                      \end{cases}
\end{equation*} 
We observe the following {power-type} asymptotics of $p$: 
\begin{equation}\label{ptxyEstimates}
      p(s,x,t,y) \approx
        \frac{t-s}{|y-x|^{d+1}} \land (t-s)^{-d}\,,\quad x,y\in  \reald\,,\; s<t\,,
\end{equation}
where $L\approx R$ means that $L/R$ is bounded away from zero and infinity.
In consequence, there  is a constant $c$ depending only on $d$, such that
\begin{eqnarray}\label{3P:ineq}
        p(s,x,u,z) \land p(u,z,t,y) \le c\, p(s,x,t,y)\,,\quad x,z,y \in  \reald\,,\;s,u,t\in \real\,,     
\end{eqnarray}         
see the 3P Theorem in \cite{MR2283957}.
For numbers $a,b\geq 0$ we have $ab = (a \vee b) (a \land b)$ and
$a\vee b \le a+b$.                                                                              
Therefore (\ref{3P:ineq}) yields
the following variant:                                                                           
\begin{equation}\label{5P:ineq}
p(s,x,u,z)p(u,z,t,y)\le c\, p(s,x,t,y) \big[p(s,x,u,z) + p(u,z,t,y)\big]\,,
\end{equation}
and we obtain 
$$
p_1(s,x,t,y)\leq c\, p(s,x,t,y)\int_\reald\int_s^t \left[p(s,x,u,z)+ p(u,z,t,y)\right]d\mu(u,z)\,.
$$

Assume that $\mu$ is of Kato class, to wit,
$$
k(h):=\sup_{x,y\in \reald,\; s<t\leq s+h} \int_\reald\int_s^t \left[p(s,x,u,z)+ p(u,z,t,y)\right]d\mu(u,z) \to 0 \quad \mbox{ as } h\to 0\,.
$$
Let $h>0$ and $\eta:=ck(h)<1$. If $s+(j-1)h<t\leq s+jh$, where $j$ is a natural number, then,  by Corollary~\ref{main-semick},  for all $x,y\in \reald$,
\begin{equation*}
p^\mu(s,x,t,y)
\le \left(\frac{1}{1-\eta}\right)^j p(s,x,t,y)
\le \left(\frac{1}{1-\eta}\right)^{1+(t-s)/h} p(s,x,t,y)\,.
\end{equation*}
This is a special case of \cite[Theorem~1]{MR2507445}. 
In particular, if $d>1$, then,  by (\ref{ptxyEstimates}), 
$$\int_s^{t} p(s,x,u,z)du\approx |z-x|^{1-d}\wedge
\big[(t-s)^2|z-x|^{-d-1}\big]\,,\quad x,y\in \reald\,,\; s<t\,,
$$
and if $|d\mu(u,z)|\leq
|z|^{-1+\varepsilon}dzdu$ for some $\varepsilon\in (0, 1]$, then $\mu$ is of Kato class.
}
\end{example}
We refer the reader to \cite{MR2457489} for a comparison of different Kato conditions. We also refer to \cite{MR2471145} for a discussion of discontinuous multiplicative functionals of Markov processes, which 
bring some analogies with Example~\ref{ex:at2}. We also wish to mention recent results \cite{2012arXiv1205.4571B, arXiv:1112.3401} for non-local Schr\"odinger-type perturbations (see \cite{MR2316878} and \cite{MR2446046}, too). 
Schr\"odinger perturbations of  the Gaussian transition density 
are studied in \cite{MR1978999,  MR2253015}, see also \cite{MR2395164}.
We refer to \cite{MR2207878,MR2160104,MR2457489,MR2283957,MR2643799} for further instances, applications and forms of the 3P (or 3G) inequality (\ref{3P:ineq}).
In a related paper \cite{2011-KBTJSS} we present a more specialized approach to Schr\"odinger perturbations by functions for
transition densities, transition probabilities and general integral kernels in continuous time.

\bibliographystyle{abbrv}
%\bibliography{loc-bib}
\def\polhk#1{\setbox0=\hbox{#1}{\ooalign{\hidewidth
  \lower1.5ex\hbox{`}\hidewidth\crcr\unhbox0}}}

\end{document}